\documentclass [12pt]{amsart}
\usepackage[utf8]{inputenc}
\pdfoutput=1

\usepackage{Andrew}

\usepackage{graphicx}
\usepackage{amsmath,amssymb,amsthm,amsfonts}

\usepackage{hyperref}%

\hypersetup{bookmarksdepth=2}
\usepackage{mathrsfs}%
\usepackage{subfig}%

\usepackage{xcolor}%
\usepackage{xspace}

\usepackage[margin=1.0in]{geometry}%
\usepackage{enumitem}

\usepackage{hyperref}

\usepackage{epstopdf}%

\newtheorem{theorem}{Theorem}[section]

\newtheorem{lemma}[theorem]{Lemma}

\theoremstyle{definition}

\theoremstyle{definition}
\newtheorem{remark}[theorem]{Remark}
\theoremstyle{definition}
\newtheorem{definition}[theorem]{Definition}

\newtheorem{question}[theorem]{Question}
\theoremstyle{definition}

\theoremstyle{definition}

\newcommand{\precdot}{\prec\mathrel{\mkern-5mu}\mathrel{\cdot}}

\def\conv{\operatorname{conv}}
\def\diam{\operatorname{diam}}
\def\des{\operatorname{des}}

\def\stat{\operatorname{stat}}
\def\outdegree{\operatorname{outdegree}}

\title{A Realization of Poset Associahedra}
\author{Andrew Sack}
\address{Department of Mathematics, University of California, Los Angeles, CA 90095, USA}
\email{{\href{mailto:andrewsack@math.ucla.edu}{andrewsack@math.ucla.edu}}}
\thanks{This material is based upon work supported by the National Science Foundation
Graduate Research Fellowship Program under Grant No. DGE-2034835 and National Science Foundation Grants No. DMS-1954121 and DMS-2046915. Any
opinions, findings, and conclusions or recommendations expressed in this material
are those of the author(s) and do not necessarily reflect the views of the National
Science Foundation.}
\date{\today}

\usepackage[backend=bibtex]{biblatex}
\begin{document}


\begin{abstract}
Given any connected poset $P$, we give a simple realization of Galashin's poset associahedron $\mathscr A(P)$ as a convex polytope in $\mathbb R^P.$  The realization is inspired by the description of $\mathscr A(P)$ as a compactification of the configuration space of order-preserving maps~$P \to \mathbb{R}.$  In addition, we give an analogous realization for Galashin's affine poset cyclohedra.
\end{abstract}

\maketitle

\section{Introduction}

Given a finite connected poset $P$, the poset associahedron $\mathscr A(P)$ is a simple, convex polytope of dimension $|P|-2$ introduced by Galashin~\cite{galashin2021poset}.  Poset associahedra arise as a natural generalization of Stasheff's associahedra~\cite{haiman1984constructing, petersen2015Eulerian, StasheffCyclohedron, tamari1954monoides}, and were originally discovered by considering compactifications of the configuration space of order-preserving maps~${P\to\mathbb R.}$   These compactifications are generalizations of the Axelrod\nobreakdash--Singer compactification of the configuration space of points on a line~\cite{axelrod1994chern, lambrechts2010associahedron, sinha2004manifold}.  Galashin constructed poset associahedra by performing stellar subdivisions on the polar dual of Stanley's \emph{order polytope}~\cite{stanley1986two}, but did not provide an explicit realization.  Various poset associahedra and cyclohedra have already been studied including \emph{permutohedra}, \emph{associahedra}, \emph{operahedra}~\cite{laplante2022diagonal}, \emph{type B permutohedra}~\cite{fomin2005root}, and \emph{cyclohedra}~\cite{bott1994self}. 

Poset associahedra bear resemblance to graph associahedra, where the face lattice of each is described by a \emph{tubing criterion.}  However, neither class is a subset of the other. When Carr and Devadoss introduced graph associahedra in~\cite{carr2006coxeter}, they distinguish between \emph{bracketings} and \emph{tubings} of a path, where the idea of bracketings does not naturally extend to any simple graph.  In the case of poset associahedra, the idea of bracketings \emph{does} extend to every connected poset.

Galashin~\cite{galashin2021poset} also introduces \emph{affine posets,} and analagous \emph{affine order polytopes} and \emph{affine poset cyclohedra}. In this paper, we provide a simple realization of poset associahedra and affine poset cyclohedra as an intersection of half spaces, inspired by the compactification description and by a similar realization of graph associahedra due to Devadoss~\cite{devadoss2009realization}.   In independent work~\cite{mantovani2023Poset}, Mantovani, Padrol, and Pilaud found a realization of poset associahedra as sections of graph associahedra.  The authors of~\cite{mantovani2023Poset} also generalize from posets to oriented building sets (which combine a building set with an oriented matroid).

\pagebreak
\section{Background}

\subsection{Poset Associahedra}

We start by defining the poset associahedron.

\begin{definition}
\label{def:tubings}
Let $(P, \preceq)$ be a finite poset.  We make the following definitions:
\begin{itemize}
    \item A subset $\tau \subseteq P$ is \emph{connected} if it is connected as an induced subgraph of the Hasse diagram of $P$.
    \item $\tau \subseteq P$ is \emph{convex} if whenever $a, c \in \tau$ and $b \in P$ such that $a \preceq b \preceq c$, then $b \in \tau$.
    \item A \emph{tube} of $P$ is a connected, convex subset $\tau \subseteq P$ such that $2 \le |\tau|$.  
    \item A tube $\tau$ is \emph{proper} if $|\tau| \le |P|-1.$
    \item Two tubes $\sigma, \tau \subseteq P$ are \emph{nested} if $\sigma \subseteq \tau$ or $\tau \subseteq \sigma.$ Tubes $\sigma$ and $\tau$ are \emph{disjoint} if $\tau \cap \sigma = \emptyset$.
    \item For disjoint tubes $\sigma, \tau$ we say $\sigma \prec \tau$ if there exists $a \in \sigma, b \in \tau$ such that $a \prec b.$
    \item A \emph{proper tubing} $T$ of $P$ is a set of proper tubes of $P$ such that any pair of tubes is nested or disjoint and the relation $\prec$ extends to a partial order on $T$.  That is, whenever $\tau_1, \dots, \tau_k \in T$ with $\tau_1 \prec \dots \prec \tau_k$ then $\tau_k \not\prec \tau_1$.  This is referred to as the \emph{acyclic tubing condition.}
    \item A proper tubing $T$ is $\emph{maximal}$ if it is maximal by inclusion on the set of all proper tubings.
    \end{itemize}

\begin{figure}
  \makebox[1.0\textwidth]{
\scalebox{0.94}{
\ \ \quad\begin{tabular}{cccccc|cccccc}
 \includegraphics{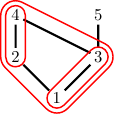}& \qquad&
  \includegraphics{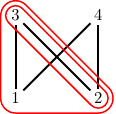}& \qquad&
  \includegraphics{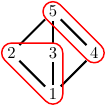}
& \qquad&\qquad & 
  \includegraphics{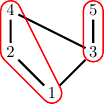}& \qquad&
  \includegraphics{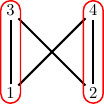}& \qquad&
  \includegraphics{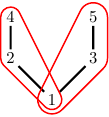}\\ &&&&&&&&&&& \\
& & Examples &&&&& & & Non-examples & &
\end{tabular}
}
}
  \caption{\label{fig:tubing_examples}Examples and non-examples of proper tubings.}
\end{figure}

Figure \ref{fig:tubing_examples} shows examples and non-examples of proper tubings.

\end{definition}

\begin{definition}
For a finite poset $P$, the \emph{poset associahedron} $\mathscr A(P)$ is a simple, convex polytope of dimension $|P|-2$ whose face lattice is isomorphic to the set of proper tubings ordered by reverse inclusion.  That is, if $F_T$ is the face corresponding to $T$, then $F_S \subset F_T$ if one can make $S$ from $T$ by adding tubes.  Vertices of $\mathscr A(P)$ correspond to maximal tubings of $P$.
\end{definition}

We realize poset associahedra as an intersection of half-spaces.  Let $P$ be a finite poset and let $n = |P|$.  We work in the ambient space $\mathbb R^P_{\Sigma = 0}$, the space of real\nobreakdash-valued functions on $P$ that sum to $0$.  For a subset $\tau \subseteq P$, define a linear function $\alpha_\tau$ on $\mathbb R^P_{\Sigma = 0}$ by $$\alpha_\tau(p) := \sum\limits_{\substack{i \precdot j \\ i, j \in \tau}} p_j - p_i.$$  
Here the sum is taken over all covering relations contained in $\tau$.  We define the half-space $h_\tau$ and the hyperplane $H_\tau$ by

$$\begin{aligned}
h_\tau &:= \{p \in \mathbb R^P_{\Sigma = 0} \mid \alpha_\tau(p) \ge n^{2|\tau|}\} &\text{ and }
\\H_\tau &:= \{p \in \mathbb R^P_{\Sigma = 0} \mid \alpha_\tau(p) = n^{2|\tau|}\}. & 
\end{aligned}$$

The following is our main result in the finite case:
\begin{theorem}\label{thm:main_thm_finite}
If $P$ is a finite, connected poset, the intersection of $H_P$ with $h_\tau $ for all proper tubes $\tau$ gives a realization of $\mathscr A(P)$.
\end{theorem}

\subsection{Affine Poset Cyclohedra}

Now we describe affine poset cyclohedra.  

\begin{definition}
An \emph{affine poset} of \emph{order} $n \ge 1$ is a poset $\tilde P = (\Z, \preceq)$ such that:
\begin{enumerate}
\item For all $i \in \Z, i \preceq i+n$;
\item $\tilde P$ is $n$-periodic: For all $i, j \in \Z, i \preceq j \Leftrightarrow i + n \preceq j + n$;
\item $\tilde P$ is \emph{strongly connected}: for all $i, j \in \Z$, there exists $k \in \Z$ such that $i \preceq j + kn$.
\end{enumerate}

The \emph{order} of $\tilde P$ is denoted $|\tilde P| := n$.

\end{definition} 

Following Galashin~\cite{galashin2021poset}, we give analagous versions of Definition \ref{def:tubings}.  We give them only where they differ from the finite case. 
\begin{definition}
Let $\tilde P = (\Z, \preceq)$ be an affine poset.
\begin{itemize}
    \item A \emph{tube} of $\tilde P$ is a connected, convex subset $\tau \subseteq P$ such that $2 \le |\tau|$ and either $\tau = \tilde P$ or $\tau$ has at most one element in each residue class modulo $n$. 
    \item A collection of tubes $T$ is \emph{$n$-periodic} is for all $\tau \in T, k \in \Z$, $\tau + kn \in T$.    
    \item A \emph{proper tubing} $T$ of $\tilde P$ is an $n$-periodic set of proper tubes of $\tilde P$ that satisfies the acyclic tubing condition and such that any pair of tubes is nested or disjoint.
\end{itemize}

Figure \ref{fig:aff_claw} gives an example of an affine poset of order $4$ and a maximal tubing of that poset.

\begin{figure}
  \makebox[1.0\textwidth]{
\scalebox{1.0}{
\ \ \quad\begin{tabular}{c|c}
 \includegraphics[scale = .8]{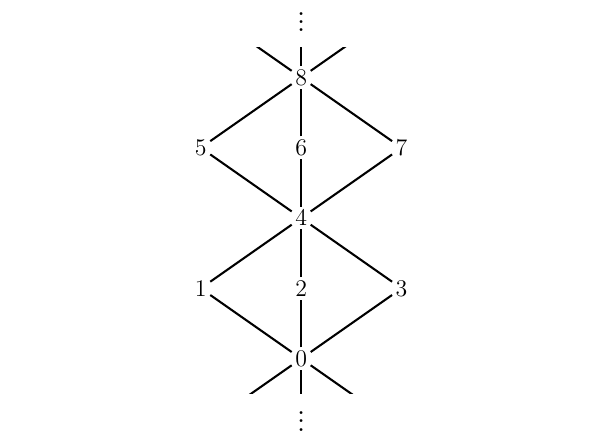}&
  \includegraphics[scale = .8]{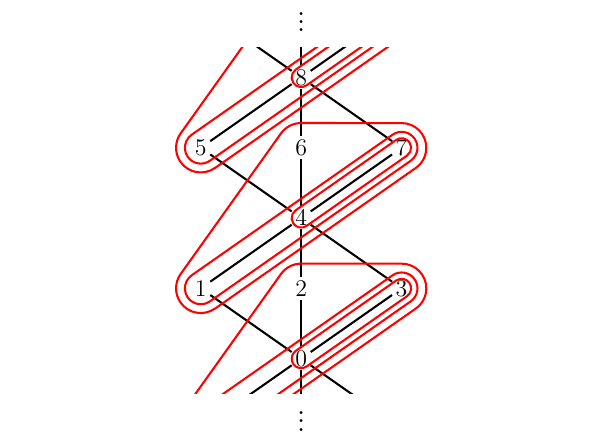}
\\

 $\tilde P$ &  A maximal tubing of $\tilde P$ 
\end{tabular}
}
}
  \caption{\label{fig:aff_claw} An affine poset of order $4$ and a maximal tubing}
\end{figure}

\end{definition}

\begin{definition}
For an affine poset $\tilde P$, the \emph{affine poset cyclohedron} $\mathscr C(\tilde P)$ is a simple, convex polytope of dimension $|\tilde P|-1$ whose face lattice is isomorphic to the set of proper tubings ordered by reverse inclusion.   Vertices of $\mathscr C(\tilde P)$ correspond to maximal tubings of $\tilde P$.
\end{definition}

We also realize affine poset cyclohedra as an intersection of half-spaces. Let $\tilde P$ be an affine poset and let $n = |\tilde P|$.  Fix some constant $c \in \R^+$.  We define the space of \emph{affine maps} $\R^{\tilde P}$ as the set of bi-infinite sequences $\mathbf{\tilde x} = (\tilde x_i)_{i \in \Z}$ such that $\tilde x_i = \tilde x_{i+n}+c$ for all $i \in \Z$.  Let $\mathcal C \subset \R^{\tilde P}$ be the subspace consisting of all constant maps.  We work in the ambient space $\R^{\tilde P}/\mathcal C$ where the constant $c$ in the definition of affine maps is given by $c = n^{2(n+1)}$.

For a finite subset $\tau \subseteq P$, define a linear function $\alpha_\tau$ on $\mathbb R^{\tilde P}/\mathcal C$ by $$\alpha_\tau(\mathbf{\tilde x}) := \sum\limits_{\substack{i \precdot j \\ i, j \in \tau}} \tilde x_j - \tilde x_i.$$  Again, the sum is taken over all covering relations contained in $\tau$.   We define the half-space $h_\tau$ and the hyperplane $H_\tau$ by

$$\begin{aligned} 
h_\tau &:= \{p \in \mathbb R^{\tilde P}/{\mathcal C} \mid \alpha_\tau(p) \ge n^{2|\tau|}\} &\text{ and }
\\H_\tau &:= \{p \in \mathbb R^{\tilde P}/{\mathcal C}  \mid \alpha_\tau(p) = n^{2|\tau|}\}.  & 
\end{aligned}$$

\begin{remark}\label{re:nperiodic}
Observe that for any tube $\tau$ and $k \in \Z$, $h_\tau = h_{\tau + kn}$. 
\end{remark}

The following is our main result in the affine case:
\begin{theorem}\label{thm:main_thm_affine}
If $\tilde P$ is an affine poset, the intersection of $h_\tau $ for all proper tubes $\tau$ gives a realization of $\mathscr C(\tilde P)$.
\end{theorem}

\subsection{An interpretation of tubings}

When $P$ is a chain, $\mathscr A(P)$ recovers the classical associahedron.  There is a simple interpretation of proper tubings that explains all of the conditions above in terms of \emph{generalized words.}

We can understand the classical associahedron as follows: Let $P = (\{1, ..., n\}, \le)$  be a chain.  We can think of the chain as a word we want to multiply together with the rule that two elements can be multiplied if they are connected by an edge.  A maximal tubing of $P$ is a way of disambiguating the order in which one performs the multiplication.  If a pair of adjacent elements $x$ and $y$ have a pair of brackets around them, they contract along the edge connecting them and replace $x$ and $y$ by their product. 

\begin{figure}
  \makebox[1.0\textwidth]{
\scalebox{0.94}{
\ \ \quad

\begin{tabular}{ccccccccccc}
&  & \includegraphics{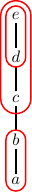} & 
\includegraphics{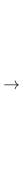} & 
\includegraphics{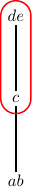} &
 \includegraphics{Graphics/GeneralizedWords/ChainTo.pdf} & 
 \includegraphics{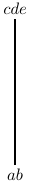} & 
 \includegraphics{Graphics/GeneralizedWords/ChainTo.pdf} &
  \includegraphics{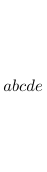} & 
        &                    \\
                   &       &                    &       &                    &       &                    &       &                    &       &                    \\
\includegraphics{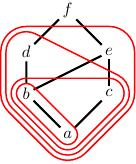} & 
\includegraphics{Graphics/GeneralizedWords/ChainTo.pdf} &
 \includegraphics{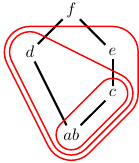} & 
\includegraphics{Graphics/GeneralizedWords/ChainTo.pdf} & 
\includegraphics{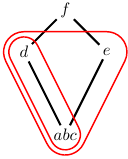} & 
\includegraphics{Graphics/GeneralizedWords/ChainTo.pdf} & 
\includegraphics{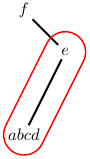} & 
\includegraphics{Graphics/GeneralizedWords/ChainTo.pdf} &
 \includegraphics{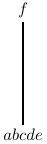} & 
\includegraphics{Graphics/GeneralizedWords/ChainTo.pdf} & 
\includegraphics{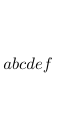}
\end{tabular}

}
}
  \caption{\label{fig:generalized_words} Multiplication of a word and of a generalized word}
\end{figure}

Similarly, we can understand the Hasse diagram of an arbitrary poset $P$ as a \emph{generalized word} we would like to multiply together.  Again, we are allowed to multiply two elements if they are connected by an edge, but when multiplying elements, we contract along the edge connecting them and then take the transitive reduction of the resulting directed graph.  That is, we identify the two elements and take the resulting quotient poset.  A maximal tubing is again a way of disambiguating the order of the multiplication. See Figure \ref{fig:generalized_words} for an illustration of this multiplication. This perspective is discussed in relation to operahedra in~\cite[Section 2.1]{laplante2022diagonal} when the Hasse diagram of $P$ is a rooted tree.

\section{Configuration spaces and compactifications}

We turn our attention to the relationship between poset associahedra and configuration spaces. For a poset $P$, the \emph{order cone} $$\mathscr L(P) := \{p \in \mathbb R^P_{\Sigma = 0} \mid p_i \le p_j \text{ for all } i \preceq j\}$$
is the set of order preserving maps $P \to \mathbb R$ whose values sum to $0$.

Fix a constant $c \in \mathbb R^+$.  The \emph{order polytope,} first defined by Stanley~\cite{stanley1986two} and extended by Galashin~\cite{galashin2021poset}, is the $(|P|-2)$-dimensional polytope $$\mathscr O(P) := \{p \in \mathscr L(P) \mid \alpha_P(p) = c\}.$$  

\begin{remark}
\label{re:stanley}
When $P$ is \emph{bounded}, that is, has a unique maximum $\hat 1$ and minimum $\hat 0$, this construction is projectively equivalent to Stanley's order polytope where we replace the conditions of the coordinates summing to $0$ and $\alpha_P(p) = c$ with the conditions $p_{\hat 0} = 0$ and $p_{\hat 1} = 1$, see~\cite[Remark~2.5]{galashin2021poset}.
\end{remark}

Galashin~\cite{galashin2021poset} obtains the poset associahedra by an alternative compactification of $\mathscr O^\circ(P)$, the interior of $\mathscr O(P)$.  We describe this compactification informally, as it serves as motivation for the realization in Theorem \ref{thm:main_thm_finite}.

A point is on the boundary of $\mathscr O(P)$ when any of the inequalities in the order cone achieve equality.  The faces of of $\mathscr O(P)$ are in bijection with proper tubings of $P$ such that all tubes are disjoint.  Let $T$ be such a tubing.  If $p$ is in the face corresponding to $T$ and $\tau \in T$ then $p_i = p_j$ for $i, j \in \tau$.

We can think of the point $p$ in the face corresponding to $T$ as being ``what happens in $\mathscr O(P)$'' when for each $\tau \in T$, the coordinates are infinitesimally close.  However, by taking all coordinates in $\tau$ to be equal, we lose information about their relative ordering.  In $\mathscr A(P)$, we still think of the coordinates in $\tau$ as being infinitesimally close, but we are still interested in their configuration.  Upon zooming in, this is parameterized by the order polytope of the subposet $(\tau, \preceq)$.  We iterate this process, allowing points in $\tau$ to be infinitesimally closer, and so on.  We illustrate this in Figure \ref{fig:compactification}.  This idea is a common explanation of the Axelrod\nobreakdash--Singer compactification of $\mathscr O^\circ(P)$ when $P$ is a chain, see~\cite{axelrod1994chern, lambrechts2010associahedron, sinha2004manifold}.

\begin{figure}
  \setlength{\tabcolsep}{0pt}
\makebox[1.0\textwidth]{
\scalebox{0.94}{
\begin{tabular}{cccc|cccc}

\includegraphics{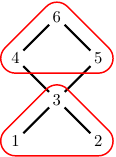}
&\qquad\qquad
&\includegraphics{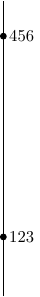}

&\qquad & \qquad&

\includegraphics{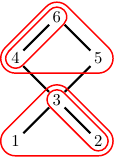}
&\qquad\qquad
&\includegraphics{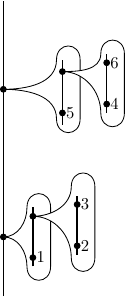}

\\

Tubing in $\mathscr O(P)$
& \qquad\qquad
& Point in $\mathscr O(P)$

& \qquad\qquad & \qquad\qquad &

Tubing in $\mathscr A(P)$
& \qquad\qquad
& Point in $\mathscr A(P)$

\end{tabular}
}
}
  \caption{\label{fig:compactification} A vertex in $\mathscr O(P)$ vs. $\mathscr A(P).$}
\end{figure}

The idea of the realization in Theorem \ref{thm:main_thm_finite} is to replace the notions of \emph{infinitesimally close} and \emph{infinitesimally closer} with being \emph{exponentially close} and \emph{exponentially closer.}  For $p \in \mathscr L(P)$, $\alpha_\tau$ acts a measure of how close the coordinates of $p|_\tau$ are. We can make this precise with the following definition and lemma.  

\begin{definition}\label{def:diameter}
For $S \subseteq P$ and $p \in \mathbb R^P$, define the \emph{diameter} of $p$ relative to $S$ by $$\diam_S(p) = \max\limits_{i, j \in S} |p_i - p_j|.$$  That is, $\diam_S(p)$ is the diameter of $\{p_i : i \in S\}$ as a subset of $\mathbb R$.
\end{definition}

\begin{lemma}
\label{le:alpha_bound}

Let $\tau \subseteq P$ be a tube and let $p \in \mathscr L(P)$.  Then $$\diam_\tau(p) \le \alpha_\tau(p) \le \frac{n^2}{4}\diam_\tau(p).$$
\end{lemma}
\begin{proof}
By the triangle inequality and as $\tau$ is connected, $\diam_\tau(p) \le \alpha_\tau(p)$.  For the other inequality, $$\begin{aligned}
\alpha_\tau(p) &= \sum\limits_{\substack{i \precdot j \\i, j \in \tau}} p_j - p_i  
\\&\le \sum\limits_{\substack{i \precdot j \\i, j \in \tau}}  \diam_\tau(p)
\\&\le \frac{1}{4}n^2 \diam_\tau(p)
\end{aligned}$$

The inequality in the last line comes from the fact that there are at most $\frac{n^2}{4}$ covering relations in $P$, which follows from Mantel's Theorem and the fact that Hasse diagrams are triangle\nobreakdash-free.

\end{proof}

In particular, for $p \in \mathscr L(P)$, if $p \in H_\tau$, then $\{p_i \mid i \in \tau\}$ is clustered tightly together compared to any tube containing $\tau$.  If $p \in h_\tau$, then $\{p_i \mid i \in \tau\}$ is spread far apart compared to any tube contained in $\tau$.

\section{Realizing poset associahedra}

We are now prepared to prove Theorem \ref{thm:main_thm_finite}.  Define $$\mathscr A(P) := \bigcap\limits_{\sigma \subset P} h_\sigma \cap H_P$$ where the intersection is over all tubes of $P$.  Note that $\mathscr A(P) \subseteq \mathscr L(P)$ as if $i \precdot j$ is a covering relation, then for $p \in h_{\{i, j\}}$, $p_i \le p_j$.

Theorem \ref{thm:main_thm_finite} follows as a result of three lemmas:

\begin{lemma}\label{le:vertices}
If $T$ is a maximal tubing, then $$v^T := \bigcap\limits_{\tau \in T \cup \{P\}} H_\tau $$ is a point.  

\end{lemma}

\begin{lemma}\label{le:incompatible}
If $T$ is a collection of tubes that do not form a proper tubing, then $$\bigcap\limits_{\tau \in T} H_\tau \cap \mathscr A(P) = \emptyset.$$

\end{lemma}

\begin{lemma}\label{le:interior}
If $T$ is a maximal tubing and $\tau \notin T$ is a proper tube, then $\alpha_\tau(v^T) > n^{2|\tau|}.$  That is, $v^T$ lies in the interior of $h_\tau$.
\end{lemma}

Lemma \ref{le:vertices} follows from a standard induction argument.

\begin{proof}[Proof of Lemma \ref{le:incompatible}]
If $T$ is not a collection of tubes that do proper tubing, then at least one of the following two cases holds:
\begin{enumerate}[label = (\arabic*)]
\item There is a pair of non-nested and non-disjoint tubes $\tau_1, \tau_2$ in $T$. 
\item There is a sequence of disjoint tubes $\tau_1, ..., \tau_k$ such that $\tau_1 \prec \dots \prec \tau_k \prec \tau_1$.  
\end{enumerate}

The idea of the proof is as follows: For $S \subseteq P$, define the \emph{convex hull} of $S$ as $$\conv(\sigma) := \{b \in P \mid  \exists a, c \in S : a \le b \le c \}.$$	Observe that if $p \in \mathscr L(P),$ then $\diam_S(p) \le \diam_{\conv(S)}(p)$.
Take $\sigma = \conv(\tau_1 \cup \dots \cup \tau_k)$. One can show that $\sigma$ is a tube, so Lemma \ref{le:alpha_bound} tells us that for each $\tau_i$, $\diam_{\tau_i}(p)$ is very small compared to $n^{2|\sigma|}$.  As the tubes either intersect or are cyclic, one can show this forces $\diam_{\sigma}(p)$ to also be small, so $\alpha_\sigma(p) < n^{2|\sigma|}$.

More concretely, suppose that $$p \in \bigcap H_{\tau_i} \cap \mathscr L(P) .$$
Note that for all $i$, $|\sigma| > |\tau_i| + 1$ and $\diam_{\tau_i}(p) \le n^{2(|\sigma|-1)}$. 
In case (1), let $a,b \in \sigma$.  There exists some $x \in \tau_1 \cap \tau_2$, so 
$$\begin{aligned}
|p_a -  p_b| &\le |p_a - p_x| + |p_x - p_b|
\\&\le \diam_{\tau_1}(p) + \diam_{\tau_2}(p)
\\ &\le 2n^{2(|\sigma|-1)} 
\\&<n^{2(|\sigma|)} .
\end{aligned}$$
Hence $\diam_\sigma(p) < n^{2|\sigma|}$, so by Lemma \ref{le:alpha_bound},  $p \notin h_\sigma$.

Now we move to case (2). Suppose there is a sequence of disjoint tubes $\tau_1, ..., \tau_k$ such that for each $i$ there exists $x_i, y_i \in \tau_i$ where $x_i \prec y_{i+1}$ where we take the indices $\text{mod }{k}$.
Then:
$$\begin{aligned}
  p_{y_i} - \diam_{\tau_i}(p) &\le p_{x_i}
\\p_{x_i} &\le p_{y_{i+1}}
\\p_{y_{i+1}} &\le p_{x_{i+1}} + \diam_{\tau_{i+1}}
\end{aligned}$$

Furthermore, since $\tau_i \tand \tau_{i+1}$ are disjoint, $|\tau_i| \le |\sigma| - 2$ and $\diam_{\tau_i} \le n^{2(|\sigma|-2)}$. Combining these we get $$p_{y_i} \le p_{y_{i+1}} + 2n^{2(|\sigma|-2)}.$$
Then we have: 
$$\begin{aligned} p_{y_1} & \le p_{y_{i}} + 2in^{2(|\sigma|-2)} & \text{ and }\\
p_{y_{i}} + 2in^{2(|\sigma|-2)}  &\le p_{y_1} + 2(k+1) n^{2(|\sigma|-2)}. \end{aligned}$$
These yield
$$\begin{aligned}
 p_{y_1} - p_{y_i} &\le 2in^{2(|\sigma|-2)} & \text{ and }\\
  p_{y_i} - p_{y_1} &\le 2(k - i+1) n^{2(|\sigma|-2)}.
  \end{aligned}$$  As $i, k-i+1 \le k \le \frac{n}{2}$, we have $|p_{y_1} - p_{y_i}| \le n^{2(|\sigma|-1)}$.  Finally, if $z_i \in \tau_i, z_j \in \tau_j$, then 
  $$\begin{aligned}
  |p_{z_i} - p_{z_j}| &\le |p_{z_i} - p_{y_i}| +|p_{y_i} - p_{y_1}| + |p_{y_1} - p_{y_j}|  + |p_{y_j} - p_{z_j}| \\
  &\le 4 n^{2(|\sigma|-1)} \\
  &< n^{2|\sigma|}.
  \end{aligned}$$

Hence $\diam_{\sigma} (p) < n^{2|\sigma|}$, and by Lemma \ref{le:alpha_bound},  $p \notin h_\sigma$.

\end{proof}

\begin{proof}[Proof of Lemma \ref{le:interior}]

Let $T$ be a maximal tubing of $P$ and let $\tau \notin T$ be a tube.  Define the \emph{convex hull} of $\tau$ \emph{relative} to $T$ by $$\conv_T(\tau) := \min\{\sigma \in T \mid \tau \subset \sigma\}.$$  Let $\sigma = \conv_T(\tau)$.
$T$ partitions $\sigma$ into a lower set $A$ and an upper set $B$ where $A$ and $B$ are either tubes or singletons. Furthermore, $A$ and $B$ both intersect $\tau$.  See Figure \ref{fig:Interior_Lemma_Sketch} for an example illustrating this.

\begin{figure}
  \makebox[1.0\textwidth]{
\scalebox{1}{
\ \ \quad\begin{tabular}{cc|cc}
 \includegraphics{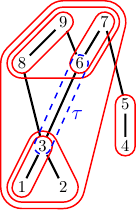}& \qquad &

\qquad & 
  \includegraphics{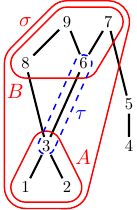}\\ &&& \\
 Maximal Tubing $T$ and tube $\tau$ &&&  $\sigma, A,$ and $B$ labelled  
\end{tabular}
}
}
  \caption{\label{fig:Interior_Lemma_Sketch}An example illustrating the proof of Lemma \ref{le:interior}.}
\end{figure}

The idea of the proof is as follows:
Let $p = v^T$.  By Lemma \ref{le:alpha_bound}, $\diam_A(p)$ and $\diam_B(p)$ are both very small compared to $\diam_\sigma(p)$.  Then for any $a \in A, b \in B$, $|p_a - p_b|$ must be large.  As $\tau$ intersects both $A$ and $B$, $\diam_\tau(p)$ must be large and hence $p \in h_\tau$.  See Figure \ref{fig:Interior_Lemma_Diameter} for an illustration of this. More precisely, we show that for any $i \in A, j \in B$, $p_j - p_i > (n^2)^{|\tau|}$, which implies that $p$ lies in the interior of $h_\tau$.

\begin{figure}

\makebox[1.0\textwidth]{
\scalebox{2}{
\includegraphics{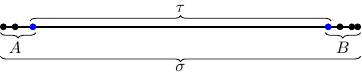}
}}

  \caption{\label{fig:Interior_Lemma_Diameter} If $\diam_A(p)$ and $\diam_B(p)$ are small and $\diam_\sigma(p)$ is large, then $\diam_\tau(p)$ is large.}

\end{figure}

 Observe that:
$$\sum\limits_{x \precdot y} p_y - p_x = 
\underbrace{\sum\limits_{\substack{x \precdot y \\ x, y \in A}} (p_y - p_x)}_{\substack{\le (n^2)^{|\sigma|-1}\\ < \frac18 (n^2)^{|\sigma|}}}+ 
\underbrace{\sum\limits_{\substack{x \precdot y \\ x, y \in B}} (p_y - p_x )}_{\substack{\le (n^2)^{|\sigma|-1}\\ < \frac18 (n^2)^{|\sigma|}}}+
 \sum\limits_{\substack{x \precdot y  \\ x \in A, y \in B}} (p_y - p_x).$$ 


Fix $i \in A \tand j \in B$.  By Lemma \ref{le:alpha_bound}, for any $x \in A, y \in B$, 
$$\begin{aligned}
p_y - p_x &\le p_j - p_i + \diam_A(p) + \diam_B(p)
\\& \le p_j + p_i + 2n^{2(|\sigma|-1)}.
\end{aligned}$$ 
Again, noting that the number of covering relations in $\sigma$ is at most $\frac{n^2}{4}$ we obtain:

$$\begin{aligned} \sum\limits_{\substack{x \precdot_\sigma y  \\ x \in A, y \in B}} (p_y - p_x) &\le  \sum\limits_{\substack{x \precdot_\sigma y  \\ x \in A, y \in B}} (p_j - p_i + 2(n^2)^{|\sigma|-1})
\\&\le \frac{n^2}{4}\left( p_j - p_i + 2(n^2)^{|\sigma|-1} \right)
\\&= \frac{n^2}{4}(p_j - p_i) + \frac12(n^2)^{|\sigma|}.
\end{aligned}$$

Combining all of this we get:

$$
\begin{aligned}
\sum\limits_{x \precdot_\sigma y} p_y - p_x &= (n^2)^{|\sigma|}
\\& < \frac{1}{8} (n^2)^{|\sigma|} + \frac{1}{8} (n^2)^{|\sigma|} + \frac{1}{2} (n^2)^{|\sigma|} +  \frac{n^2}{4}(p_j - p_i) 
\\&\le \frac{3}{4}(n^2)^{|\sigma|} + \frac{n^2}{4}(p_j - p_i) 
\end{aligned}
$$

Then $(n^2)^{|\sigma|-1} < (p_j - p_i)$ and as $|\tau| \le |\sigma| - 1$, $p$ is in the interior of  $h_\tau$.

\end{proof}

\begin{remark}
A similar approach for realizing graph associahedra is taken by Devadoss~\cite{devadoss2009realization}. For a graph $G=(V,E)$, Devadoss realizes the graph associahedron of $G$ by taking the supporting hyperplane for a graph tube $\tau$ to be $$\left\{p \in \mathbb R^V \mid \sum\limits_{i \in \tau} p_i = 3^{|\tau|}\right\}.$$
One difference is that Devadoss realizes graph associahedra by cutting off slices of a simplex whereas we cut off slices of an order polytope. When the Hasse diagram of $P$ is a tree, the poset associahedron is combinatorially equivalent to the graph associahedron of the line graph of the Hasse diagram. In this case, the two realizations have linearly equivalent normal fans. 
If the Hasse diagram of $P$ is a path graph, then both realizations have linearly equivalent normal fans to the realization of the associahedron due to Shnider and Sternberg~\cite{StasheffCyclohedron}.
\end{remark}

\section{Realizing affine poset cyclohedra}

The proofs in the affine case are nearly identical to the finite case with some additional technical components.  The similarity comes from the fact that Lemma \ref{le:alpha_bound} still applies.  We highlight where the proofs are different.   Let $\tilde P$ be an affine poset of order $n$. 

Define 
$$\begin{aligned}
\mathscr C(\tilde P) &:= \bigcap\limits_{\sigma \subset P} h_\sigma  & \text{ and }\\
\mathscr L(\tilde P) &:=  \{p \in \mathbb R^{\tilde P}/\mathcal C \mid p_i \le p_j \text{ for all } i \preceq j\}. & 
\end{aligned}$$ where the intersection is over all tubes of $\tilde P$.  Note that $\mathscr C(\tilde P) \subseteq \mathscr L(\tilde P)$ as if $i \precdot j$ is a covering relation, then for $p \in h_{\{i, j\}}$, $p_i \le p_j$.
Theorem \ref{thm:main_thm_affine} follows as a result of 3 lemmas:

\begin{lemma}\label{le:vertices_affine}
If $T$ is a maximal tubing, then $$v^T := \bigcap\limits_{\tau \in T} H_\tau $$ is a point.  
\end{lemma}

\begin{lemma}\label{le:incompatible_affine}

If $T$ is a collection of tubes that do not form a proper tubing, then $$\bigcap\limits_{\tau \in T} H_\tau \cap \mathscr C(\tilde P) = \emptyset.$$

\end{lemma}

\begin{lemma}\label{le:interior_affine}
If $T$ is a maximal tubing and $\tau \notin T$ is a proper tube, then $\alpha_\tau(v^T) > n^{2|\tau|}.$  That is, $v^T$ lies in the interior of $h_\tau$.
\end{lemma}

\begin{proof}[Proof of Lemma \ref{le:vertices_affine}]
Let $T$ be a maximal tubing and take any $\sigma \in T$ such that $|\tau| = n$.  Then restricting to $\tilde P|_{\sigma}$, Lemma \ref{le:vertices} implies that $$\bigcap\limits_{\substack{\tau \in T \\ \tau\subseteq \sigma}} H_\tau$$ is a point.  However, as $T$ is $n$-periodic, $$\bigcap\limits_{\substack{\tau \in T \\ \tau\subseteq \sigma}} H_\tau = \bigcap\limits_{\substack{\tau \in T}} H_\tau .$$ 
\end{proof}

\begin{proof}[Proof of Lemma \ref{le:incompatible_affine}]
By Remark \ref{re:nperiodic}, we can assume $T$ is $n$-periodic.  The proof is almost identical to the proof of Lemma \ref{le:incompatible}.  Define $$\mathscr L(\tilde P) :=  \{p \in \mathbb R^{\tilde P}/\mathcal C \mid p_i \le p_j \text{ for all } i \preceq j\}.$$
and note that $$\mathscr L(\tilde P) \subseteq R^{\tilde P}/\mathcal C \bigcap\limits_{\substack{i, j \in \tilde P \\ i \precdot j}} h_{\{i, j\}}.$$

Let  $$p \in \bigcap H_{\tau_i} \cap \mathscr L(\tilde P) .$$ We again break into two cases: 
\begin{enumerate}[label = (\arabic*)]
\item There is a pair of non-nested and non-disjoint tubes $\tau_1, \tau_2$ in $T$. 
\item All tubes in $T$ are pairwise nested or disjoint and there is a sequence of disjoint tubes $\tau_1, ..., \tau_k$ such that $\tau_1 \prec \dots \prec \tau_k \prec \tau_1$.  
\end{enumerate}

The only difference in the proof occurs in case (1).  Here, it is possible that there exists $x \in \tau_1 \cup \tau_2$ such that $x + n \in \tau_1 \cup \tau_2$ as well.  In this case, the proof of Lemma \ref{le:incompatible} still implies that $\diam_{\tau_1 \cup \tau_2}(p) \le \diam_{\tau_1}(p) + \diam_{\tau_2}(p) \le 2n^{2n}$.  However, $|p_{x+n} - p_{x}| = n^{2(n+1)}$.

\end{proof}

\begin{proof}[Proof of Lemma \ref{le:interior_affine}]
Let $T$ be a maximal tubing and $\tau \notin T$ be a proper tube.  Let $p = v^T$.  We claim that  $\alpha_\tau(p) > n^{2|\tau|}.$

The only difference from the proof of Lemma \ref{le:interior} is that $\tau$ may not be contained by any tube in $\tau$ so $\conv_T(\tau)$ may not be well-defined.  In this case, there exists $A \in T$ such that $|A| = n$, $A \cap \tau \neq \emptyset, \tand (A+n) \cap \tau \neq \emptyset$.  Here, $(A+n)$ acts the same as $B$ in the finite case, except the argument is much simpler.  

Let $i \in A \cap \tau, j \in (A+n) \cap \tau$.   Observe that $\diam_A(p), \diam_{(A+n)}(p) \le n^{2n}$ and that $i + n \in (A+n)$.  Then 
$$\begin{aligned}
|p_j - p_i| &\ge (p_j - n^{2n}) - p_i
\\&\ge p_{i+n} - p_i
\\&= n^{2(n+1)}.
\end{aligned}$$

Hence $\diam_{\tau}(p) > n^{2|\tau|}$ and by Lemma \ref{le:alpha_bound}, $\alpha_\tau(p) > n^{2|\tau|}$.

\end{proof}

\section{Remarks and Questions}

\begin{remark}
Let $(P, \preceq)$ be a bounded poset.  In Remark \ref{re:stanley}, we discuss how $\mathscr O(P)$ can be realized as the set of all $p \in \R^P$ such that $p_{\hat 0} = 0$, $p_{\hat 1} = 1$, and $p_i \le p_j$ whenever $i \preceq j$.  We can similarly realize $\mathscr A(P)$ as follows:  Fix $0 < \varepsilon < \frac{1}{n^2}$.

For a proper tube $\tau \subset P$, let $$h'_\tau = \{p \in \R^P \mid \alpha_\tau(p) < \varepsilon^{n-|\tau|}\}.$$
Then $\mathscr A(P)$ is realized as the intersection over all $h'_\tau$ with the hyperplanes $$\{p_{\hat 0} = 0\} \tand \{p_{\hat 1} = 1\}.$$  Letting $\varepsilon \to 0$, we obtain $\mathscr O(P)$ as a limit of $\mathscr A(P)$ as shown in Figure~\ref{fig:CubeLimit}.

\begin{figure}
  \makebox[1.0\textwidth]{
\scalebox{1.0}{
\ \ \quad\begin{tabular}{c|c|c|c}
 \includegraphics[scale = 2]{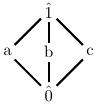}&
 \includegraphics[scale = .28]{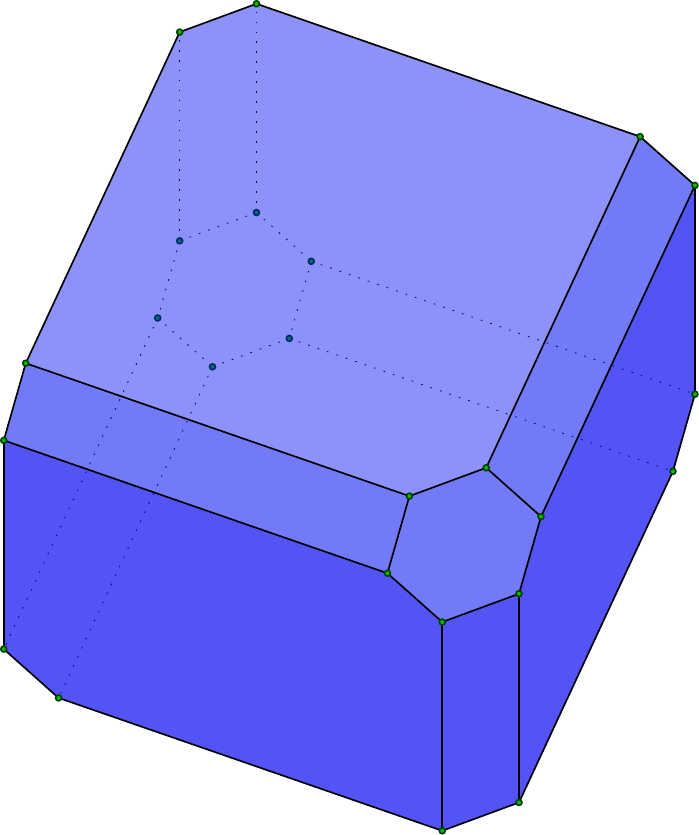}& 
 \includegraphics[scale = .28]{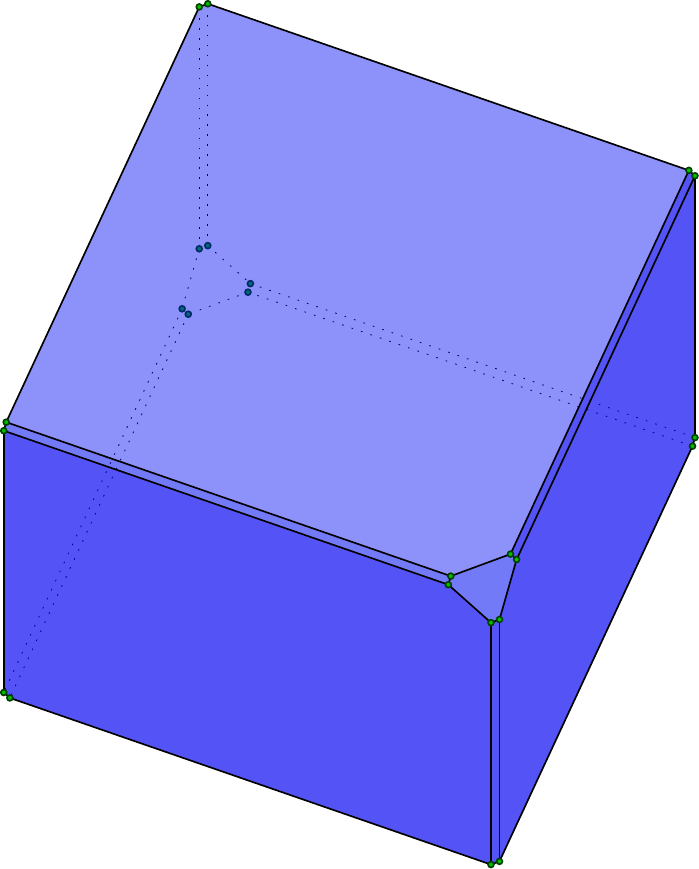}& 
 \includegraphics[scale = .28]{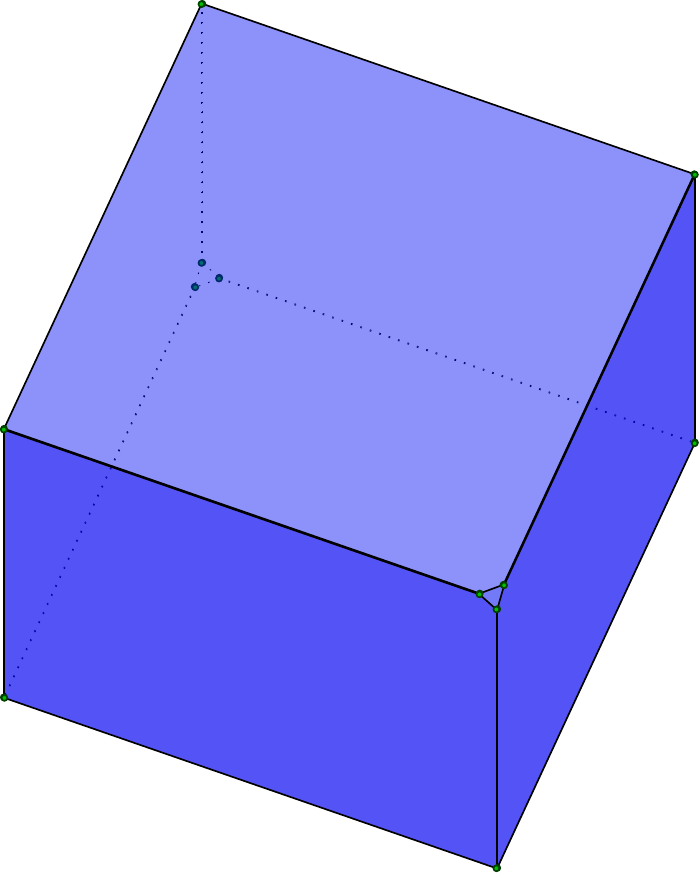}
\\

 $P$ & $\varepsilon = \frac13$ &  $\varepsilon = \frac 19$  & $\varepsilon = \frac 1{27}$
\end{tabular}
}
}
  \caption{\label{fig:CubeLimit} $\mathscr O(P)$ as a limit of $\mathscr A(P)$}
\end{figure}

\end{remark}

\begin{remark}
The key piece to the realizations in Theorems \ref{thm:main_thm_finite} and \ref{thm:main_thm_affine} is the linear form $\alpha_\tau$, where $\alpha_\tau$ acts as an approximation of $\diam_\tau$.   In particular, let $\tau$ be a tube and let $p \in \mathscr L(P)$.  Then:
\begin{itemize}
\item $\alpha_\tau(p) \ge 0$.
\item $\alpha_\tau(p) = 0 \Leftrightarrow p|_{\tau}$ is constant.
\item If $\sigma \subseteq \tau$ is a tube, then $\alpha_\sigma(p) \le \alpha_\tau(p)$.
\end{itemize}

However, there are many other options for choice of $\alpha_\tau$ that could fill this role.  Some other options include:
\begin{enumerate}
\item Sum over all pairs $i \prec j$ in $\tau$. $$\alpha_\tau(p) = \sum\limits_{\substack{i \prec j \\ i, j \in \tau}} p_j - p_i.$$
\item Let $A \tand B$ be the set of minima and maxima of the restriction $P|_{\tau}$ respectively.   
$$\alpha_\tau(p) = \sum\limits_{\substack{i \prec j \\ i \in A, j \in B}} p_j - p_i.$$
\item Fix a spanning tree $T$ in the Hasse diagram of $\tau$.  Let $E = \{(i, j) \mid i \precdot_T j\}$ be the set of edges in $T$.   
$$\alpha_\tau(p) = \sum\limits_{(i, j) \in E} p_j - p_i.$$
An advantage of this option is that we would have $$\diam_\tau(p) \le \alpha_\tau(p) \le (n-1)\diam_\tau(p).$$
\end{enumerate}
A similar realization can be obtained for each choice of of $\alpha_\tau$.  

\end{remark}

\begin{question}
\label{que:h_stat}
Recall that for a simple $d$-dimensional polytope $P$, the $f$-vector and $h$-vector of $P$ are given by $(f_0, \dots, f_d)$ \tand $(h_0, \dots, h_d)$ where $f_i$ is the number of $i$-dimensional faces and 
$$
\sum\limits_{i = 0}^d f_i t^i = \sum\limits_{i = 0}^d h_i (t+1)^i.
$$
Postnikov, Reiner, and Williams~\cite{postnikov2008faces} found a statistic on maximal tubings of graph associahedra of chordal graphs where $$\sum\limits_{T} t^{\stat(T)} =  \sum h_i t^i.$$
In particular, they define a map $T \mapsto w_T$ from maximal tubings of a graph on $n$ vertices to the set of permutations $S_n$ such that $\stat(T) = \des(w_T)$, the number of descents of $w_T$.  It would be interesting to find a similar statistic on maximal tubings of poset associahedra.  For a simple polytope $P$, one can orient the edges of $P$ according to a generic linear form and take $\stat(v) = \outdegree(v)$~\cite[\S 8.2]{ziegler2012lectures}.  It may be possible to use our realization to find the desired statistic.
\end{question}

\section*{Acknowledgements}
The author is grateful to Pavel Galashin for his many helpful comments and suggestions and to Stefan Forcey for fruitful conversations.

\printbibliography

@unpublished{mantovani2023Poset,
author  = {Mantovani, Chiara and Padrol, Arnau and Pilaud, Vincent},
title   = {Acyclonestohedra: when oriented matroids meet nestohedra},
year    = {in prep.},
}

@article{galashin2021poset,
  title={Poset associahedra},
  author={Galashin, Pavel},
  journal={arXiv preprint arXiv:2110.07257},
  year={2021}
}

@article{laplante2022diagonal,
  title={The diagonal of the operahedra},
  author={Laplante-Anfossi, Guillaume},
  journal={Advances in Mathematics},
  volume={405},
  pages={108494},
  year={2022},
  publisher={Elsevier}
}

@article{stanley1986two,
  title={Two poset polytopes},
  author={Stanley, Richard P},
  journal={Discrete \& Computational Geometry},
  volume={1},
  number={1},
  pages={9--23},
  year={1986},
  publisher={Springer}
}

@article{axelrod1994chern,
  title={Chern-Simons perturbation theory. II},
  author={Axelrod, Scott and Singer, Isadore M},
  journal={Journal of Differential Geometry},
  volume={39},
  number={1},
  pages={173--213},
  year={1994},
  publisher={Lehigh University}
}

@article{sinha2004manifold,
  title={Manifold-theoretic compactifications of configuration spaces},
  author={Sinha, Dev P},
  journal={Selecta Mathematica},
  volume={10},
  number={3},
  pages={391--428},
  year={2004},
  publisher={Springer}
}

@article{lambrechts2010associahedron,
  title={Associahedron, cyclohedron and permutohedron as compactifications of configuration spaces},
  author={Lambrechts, Pascal and Turchin, Victor and Voli{\'c}, Ismar},
  journal={Bulletin of the Belgian Mathematical Society-Simon Stevin},
  volume={17},
  number={2},
  pages={303--332},
  year={2010},
  publisher={The Belgian Mathematical Society}
}

@article{devadoss2009realization,
  title={A realization of graph associahedra},
  author={Devadoss, Satyan L},
  journal={Discrete Mathematics},
  volume={309},
  number={1},
  pages={271--276},
  year={2009},
  publisher={Elsevier}
}

@article{postnikov2008faces,
  title={Faces of Generalized Permutohedra},
  author={Postnikov, Alex and Reiner, Victor and Williams, Lauren},
  journal={Documenta Mathematica},
  volume={13},
  pages={207--273},
  year={2008}
}

@article{bott1994self,
  title={On the self-linking of knots},
  author={Bott, Raoul and Taubes, Clifford},
  journal={Journal of Mathematical Physics},
  volume={35},
  number={10},
  pages={5247--5287},
  year={1994},
  publisher={American Institute of Physics}
}

@article{StasheffCyclohedron,
author = {Stasheff, Jim},
year = {1996},
month = {09},
pages = {},
title = {From Operads to `Physically' Inspired Theories}
}

@article{tamari1954monoides,
  title={Mono{\"\i}des pr{\'e}ordonn{\'e}s et cha{\^\i}nes de Malcev},
  author={Tamari, Dov},
  journal={Bulletin de la Soci{\'e}t{\'e} math{\'e}matique de France},
  volume={82},
  pages={53--96},
  year={1954}
}

@article{haiman1984constructing,
  title={Constructing the associahedron},
  author={Haiman, Mark},
  journal={Unpublished manuscript, MIT},
  year={1984}
}

@article{carr2006coxeter,
  title={Coxeter complexes and graph-associahedra},
  author={Carr, Michael and Devadoss, Satyan L},
  journal={Topology and its Applications},
  volume={153},
  number={12},
  pages={2155--2168},
  year={2006},
  publisher={Elsevier}
}

@book{petersen2015Eulerian,
author = {Petersen, Kyle},
year = {2015},
month = {10},
pages = {},
title = {Eulerian Numbers},
isbn = {978-1-4939-3090-6},
doi = {10.1007/978-1-4939-3091-3}
}

@book{ziegler2012lectures,
  title={Lectures on polytopes},
  author={Ziegler, G{\"u}nter M},
  volume={152},
  year={2012},
  publisher={Springer Science \& Business Media}
}

@article{fomin2005root,
  title={Root systems and generalized associahedra},
  author={Fomin, Sergey and Reading, Nathan},
  journal={arXiv preprint math/0505518},
  year={2005}
}

\end{document}